\documentclass[12pt]{amsart}
\usepackage[colorlinks=true, pdfstartview=FitV, linkcolor=blue, citecolor=blue, urlcolor=blue]{hyperref}
\usepackage{amssymb,amscd}
\calclayout 

\def\quot#1#2{#1/\!\!/#2}

\def\C{\mathbb {C}}

\def\R{\mathbb {R}}

\def\N{\mathbb N}

\def\HH{\mathcal{H}}
\def\SL{\operatorname{SL}}
\def\GL{\operatorname{GL}}

\def\PSL{\operatorname{PSL}}

\def\SO{\operatorname{SO}}

\def\Diff{\mathcal D}

\def\inv{^{-1}}
\def\lie#1{{\mathfrak #1}}
\def\lieg{\lie g}
\def\lieh{\lie h}

\def\liesl{\lie {sl}}

\def\liek{\lie {k}}
\def\liet{\lie {t}}

\def\phi{\varphi}

\def\O{\mathcal O}

\def\pr{{\operatorname{pr}}}

\def\mod{\operatorname{mod}}

\def\Aut{\operatorname{Aut}}
\def\HAut{\Aut_{h}}

\def\Im{\operatorname{Im}}

\def\codim{\operatorname{codim}}

\def\rank{\operatorname{rank}}

\def\diag{\operatorname{diag}}
\def\Var{\operatorname{Var}}

\def\ql{{\operatorname{{q\ell}}}}

\def\W{\mathcal W}
\def\X{\mathfrak X}
\numberwithin{equation}{subsection}

\newtheorem{theorem}[subsection]{Theorem}
\newtheorem{lemma}[subsection]{Lemma}
\newtheorem{proposition}[subsection]{Proposition}
\newtheorem{corollary}[subsection]{Corollary}

\theoremstyle{definition}

\newtheorem{definition}[subsection]{Definition}

\newtheorem{questions}[subsection]{Questions}

\theoremstyle{remark}

\newtheorem{remark}[subsection]{Remark}

\title[Lifting automorphisms of quotients of adjoint representations]{\boldmath Lifting automorphisms of  quotients of adjoint representations} 
 \author{Gerald W. Schwarz}
\address{Department of Mathematics\\
Brandeis University\\
Waltham, MA 02454-9110}
\email{schwarz@brandeis.edu}

\begin{document}
\begin{abstract}
 Let $\lieg_i$ be a simple complex Lie algebra, $1\leq i \leq d$, and let $G=G_1\times\dots\times G_d$ be the corresponding adjoint group. Consider the $G$-module $V=\oplus r_i\lieg_i$   where $r_i\in \N$ for all $i$. We say that $V$ is \emph{large} if all $r_i\geq 2$ and  $r_i\geq 3$ if $G_i$ has rank 1.    In \cite{SchAutoms} we showed that when $V$ is large any algebraic automorphism $\psi$ of the quotient $Z:=\quot VG$ lifts to an algebraic  mapping $\Psi\colon V\to V$ which sends the fiber over $z$ to the fiber over $\psi(z)$, $z\in Z$. (Most cases were already handled in Kuttler \cite{Kuttler}.)\
 We also showed that one can choose a biholomorphic lift $\Psi$ such that $\Psi(gv)=\sigma(g)\Psi(v)$, $g\in G$, $v\in V$, where $\sigma$ is an automorphism of $G$. This leaves open  the following questions: Can one lift holomorphic automorphisms of $Z$? Which automorphisms lift if  $V$ is not large? We answer  the first question in the affirmative and also answer the second question.  Part of the proof involves establishing the following result for $V$ large. Any algebraic differential operator of order $k$ on $Z$ lifts to a  $G$-invariant algebraic differential operator of order $k$ on $V$. We also consider the analogues of the questions  above for actions of compact Lie groups.
 \end{abstract}

\maketitle

\section{Introduction}\label{sec:intro}
Our base field is $\C$, the field of complex numbers. Let $G$ be a complex reductive group and $X$  a  smooth affine $G$-variety. We denote the algebra of polynomial functions on $X$ by $\O(X)$. For the following, we refer to \cite{KraftBook}, \cite{LunaSlice} and \cite{PopovVinberg}. By Hilbert, the algebra $\O(X)^G$ is finitely generated, so that we have a quotient variety $Z:=\quot XG$ with coordinate ring $\O(Z)=\O(X)^G$. Let $\pi\colon X\to Z$ denote the morphism dual to the inclusion $\O(X)^G\subset\O(X)$. Then $\pi$ sets up a bijection between the points of $Z$ and the closed orbits in $X$. If $Gx$ is a closed orbit, then the isotropy group $H=G_x$ is reductive. The \emph{slice representation of $H$ at $x$\/} is its action on $N_x$ where $N_x$ is an $H$-complement to $T_x(Gx)$ in $T_x(X)$. Let $Z_{(H)}$ denote the points of $Z$ such that the isotropy groups of the corresponding closed orbits are in the conjugacy class $(H)$ of $H$.  The $Z_{(H)}$ give a finite stratification of $Z$ by locally closed smooth subvarieties. In particular, there is a unique open stratum $Z_{(H)}$, the \emph{principal stratum}, which we also denote by $Z_\pr$. We call  $H$ a \emph{principal isotropy group} and any associated closed orbit a \emph{principal orbit\/} of $G$.   

As shorthand for saying that $X$ has finite  principal isotropy groups we say that $X$ has FPIG. If $X$ has FPIG, then  there is an open set of closed orbits and a closed orbit is principal if and only if the slice representation of its isotropy group is trivial. Set $X_\pr:=\pi\inv(Z_\pr)$.  We say that $X$ is \emph{$k$-principal\/} if $X$ has FPIG and $\codim X\setminus X_\pr\geq k$.

For $Y$ an affine variety let $\Aut(Y)$ (resp.\ $\HAut(Y)$) denote the automorphisms (resp.\ biholomorphisms) of $Y$.  If $y\in Y$, we denote by $\HAut(Y,y)$ the germs of biholomorphisms of $Y$ which fix $y$.

Our main interest is in the following case. Let $V=\oplus_{i=1}^d r_i\lieg_i$ where the $\lieg_i$ are  simple complex Lie algebras and $r_i\lieg_i$ denotes the direct sum of $r_i$ copies of $\lieg_i$, $r_i\in\N$. Let $G_i$ denote the adjoint group of $\lieg_i$ and let $G$ denote the product of the $G_i$. We denote by $0\in Z:=\quot VG$ the image of $0\in V$. Let $\psi\in\HAut(Z)$. 
We say that  $\Psi\in\HAut(V)$ is a \emph{lift of $\psi$\/}  if $\pi\circ\Psi=\psi\circ\pi$. Equivalently,  $\Psi$   maps the fiber   $\pi\inv(z)$  to the fiber $\pi\inv(\psi(z))$, $z\in Z$. Let $\sigma$ be an automorphism of $G$. We say that $\Psi$ is \emph{$\sigma$-equivariant\/} if $\Psi(gv)=\sigma(g)\Psi(v)$ for all $v\in V$, $g\in G$. 
We say that $V$ is \emph{large\/} if 
 $r_i\geq 2$ for all $i$ and  $r_i\geq 3$ if $\lieg_i\simeq\liesl_2$.

  \begin{lemma}
  The following are equivalent:
 \begin{enumerate}
 \item $V$ is large.
\item $Z$ has no codimension one strata.
\item $V$ is $2$-principal.
\item Any $\psi\in\HAut(Z)$   preserves the stratification of $Z$.
\item No $\quot{(r_i\lieg_i)}{G_i}$ is smooth.
\end{enumerate}
 \end{lemma}

\begin{proof}
By \cite[Proposition 3.1]{SchVectorFields}, (1) and  (2) are equivalent. Since $G$ is semisimple, \cite[Corollary 7.4]{SchLifting} shows that (2) implies (3). Since $V$ is an orthogonal representation,  \cite[Theorem 1.2]{SchVectorFields} shows that (3) implies (4), which in turn clearly implies (5). Since some $\quot{(r_i\lieg_i)}{G_i}$ is smooth when $V$ is not large, (5) implies (1)
\end{proof}

The lifting problem for $V$ has been investigated in  \cite{Kuttler} and \cite{SchAutoms}. In \cite{SchAutoms} we showed that $V$ has the following  lifting property.

\begin{theorem}\label{thm:adjoint}
Let $V$  be large. Then for any $\psi\in\Aut(Z)$ there is a morphism $\Psi\colon V\to V$ such that $\pi\circ \Psi=\psi\circ\pi$.
\end{theorem} 

Most cases of the theorem  were first established by Kuttler \cite{Kuttler}. The lift given in the theorem is not necessarily an automorphism of $V$ nor is it necessarily $\sigma$-equivariant for some $\sigma$. 
From \cite{SchAutoms} we get the following 

\begin{corollary}\label{cor:largelifting}
Let $V$ and $\psi$ be as above. Then there is an automorphism $\sigma$ of $G$  and a $\sigma$-equivariant lift $\Psi\in\HAut(V)$. Hence $\psi$  sends $Z_{(H)}$ to $Z_{(\sigma(H))}$ for every stratum $Z_{(H)}$ of $Z$.
\end{corollary}
\noindent  The action of $\psi$ on the strata is obtained by a different method in \cite{Kuttler}.

\smallskip
The above results beg the following 
\begin{questions}
\begin{enumerate}
\item Can we lift elements of $\HAut(Z)$ to $\sigma$-equivariant elements of $\HAut(V)$?
\item Which automorphisms lift if $V$ is not large?
\end{enumerate}
\end{questions}
\medskip
We find necessary and sufficient conditions for a $\psi$ in $\HAut(Z)$ 
 to have a holomorphic $\sigma$-equivariant lift. The extra conditions that one needs have to do with the codimension one strata of $Z$.

Suppose that some $r_i$ is 1. Then $Z_i:=\quot {\lieg_i}{G_i}\simeq\liet_i/\W_i$ where $\liet_i$ is the Lie algebra of a maximal torus of $G_i$ and $\W_i$ is the Weyl group. Now the closures of the codimension one strata of $Z_i$ are the images of the reflection  hyperplanes in $\liet_i$. If $\lieg_i$ has two root lengths, then we get two strata of codimension one in $Z$, corresponding to the reflection hyperplanes of the short roots and the long roots. Let $D_{i, s}$ and $D_{i,\ell}$  be the corresponding divisors of $Z$.  It $\lieg_i$ is simply laced, then $D_{i,s}=\emptyset$.  
Let $D_s$ and $D_\ell$ be the union of the $D_{i,s}$ and $D_{i,\ell}$ for $r_i=1$. 

There is one other way of getting a codimension one stratum, which occurs  if $r_i=2$ and $\lieg_i\simeq\liesl_2$. Then $(2\liesl_2,\PSL_2)\simeq (2\C^3,\SO_3)$. Thus the quotient is $\C^3$ and the closure of the codimension one stratum is the zero  set  $D_i$ of a non degenerate quadratic form $xy-z^2$. The corresponding isotropy class is   $(\SO_2)$. Let $D_0$ denote the union of the $D_i$ for $r_i=2$ and $\lieg_i\simeq\liesl_2$. Then the closure  of the codimension one strata of $Z$ is $D:=D_0\cup D_s\cup D_\ell$. 

Here is our main theorem.

\begin{theorem}\label{thm:main}
Let $V$, $G$, $D$ and $D_s$ be as above. Let $\psi\in\HAut(Z)$.  
 Then $\psi$ has a $\sigma$-equivariant  biholomorphic lift for some $\sigma$ if and only if $\psi$ preserves $D$ and $D_s$.
\end{theorem}

If  $\Psi\colon V\to V$ is a   $\sigma$-equivariant biholomorphic lift of $\psi$, then   $\psi$ sends $Z_{(H)}$ to $Z_{\sigma(H)}$ for all $(H)$. Thus $\psi$ permutes the  strata, in particular the codimension one strata. Moreover, if $(H)$ corresponds to the short roots of some $\lieg_i$ (so $r_i=1$), then $(\sigma(H))$ corresponds to the short roots of some $\lieg_j$ with $r_j=1$. Hence $\psi$ has to stabilize $D_s$ and we see that the conditions of the theorem are necessary. For sufficiency we proceed in several steps. The first step, taken in \S 2, is to   show that if $\psi$ preserves  $D$, then $\psi$ \emph{preserves the stratification\/}, i.e., it permutes the strata.   If $\psi$ also preserves $D_s$, then we say that $\psi$ is \emph{strongly stratification preserving\/}.

Next we need to use an action of $\C^*$ on $Z$.  For $z=\pi(v)$, $t\cdot z=\pi(tv)$, $t\in\C^*$.  One can also see the action   as follows.  Let $p_1,\dots,p_r$ be homogeneous generators of $\O(V)^G$ and let $p=(p_1,\dots,p_r)\colon V\to\C^r$. Let $Y$ denote the image of $p$. Then we can identify $Y$ with $Z=\quot VG$. If $e_i$ is the degree of $p_i$, then we have a $\C^*$-action on $Y$ where $t\in\C^*$ sends  $(y_1,\dots,y_r)\in Y$ to $(t^{e_1}y_1,\dots,t^{e_r}y_r)$. The isomorphism $Y\simeq Z$ identifies the two $\C^*$-actions.  

We say that $\psi\in\Aut(Z)$ is \emph{quasilinear} if $\psi(t\cdot z)=t\cdot\psi(z)$ for all $z\in Z$ and $t\in\C^*$. We write $\psi\in\Aut_\ql(Z)$. A main idea, as in 
\cite{SchAutoms}, is to deform a general $\psi$ to one that is quasilinear. Assume that $\psi(0)=0$. For $t\in\C^*$, let $\psi_t(z)$ denote $t\inv\cdot\psi(t\cdot z)$. We say that $\psi\in\HAut(Z,0)$ is \emph{deformable\/} if $\psi_0(z):=\lim_{t\to 0}\psi_t(z)$ exists for all $z$ sufficiently close to $0$. Let $S=\oplus S_k$ denote the graded ring $\O(Z)$, where the grading comes from the $\C^*$-action.  Let $\HH(Z)$ denote the holomorphic functions on $Z$ and $\HH(Z,0)$ the germs of holomorphic functions at $0\in Z$. Then $\psi$ is deformable if and only if $\psi^*S_k\subset S_k\cdot\HH(Z,0)$ for all $k$. Equivalently, $\psi^*y_i\in S_{e_i}\cdot\HH(Z,0)$ for all $i$. This is a vanishing condition on some of the partial derivatives of $\psi$ at $0$. If $\psi$ is deformable, then  $\psi_0^*$ sends $y_i$ to the term homogeneous of degree $e_i$ of the Taylor series of $\psi^*y_i$ at $0$, considered as an element of $\O(Z)$. Clearly, 
if $\psi$ is deformable, the family $\psi_t$ (including $t=0$) is holomorphic in $z$ and $t$. (If $\psi=\psi(x,z)$ also depends holomorphically on  parameters $x\in X$, then the family $\psi_t(x,z)$ is holomorphic in $t$, $x$ and $z$.)\  Each $\psi_t$ preserves the germ of $Z$ at $0$ and $\psi_0$ preserves $Z$. If $\psi$ and $\psi\inv$ are deformable, then $(\psi\inv)_0$ is an inverse to $\psi_0$ so that $\psi_0\in\Aut_\ql(Z)$. If $\psi$ is also (strongly) stratification preserving, then so is  $\psi_0$ since the strata are invariant under the action of $\C^*$. We say that $V$   is \emph{good\/} if every stratification preserving $\psi\in\HAut(Z,0)$ is deformable.
We say that $V$ has the \emph{lifting property\/} if every strongly stratification preserving $\psi\in\Aut_\ql(Z)$ has a $\sigma$-equivariant lift $\Psi\in\GL(V)$ for some $\sigma$. Then  $\Psi\circ g\circ\Psi\inv=\sigma(g)$, so that  $\Psi\in N_{\GL(V)}(G)$, the normalizer of $G$ in $\GL(V)$. 

\begin{theorem}\label{thm:mainhelp}
Let $(V,G)$ be as above. Then
\begin{enumerate}
\item $V$ is good.
\item $V$ has the lifting property
\end{enumerate}
\end{theorem}
We deduce our main theorem from the one above using a result on lifting of isotopies  as in \cite{SchLifting}. Our proof of Theorem \ref{thm:mainhelp} breaks up into three parts.  Recall that  $V=\oplus_{i=1}^d r_i\lieg_i$. Let $V_1$ be the sum of the submodules $r_i\lieg_i\simeq 2\liesl_2$, let $V_2$ be the sum of the submodules $r_i\lieg_i$ with $r_i=1$ and let $V_3$ be the sum of  the $r_i\lieg_i$ which are large. Let $\widetilde G_i$ be the image of $G$ in $\GL(V_i)$, $1\leq i\leq 3$. We show that Theorem \ref{thm:mainhelp} holds for each $(V_i,\widetilde G_i)$. The first two cases are handled in \S 3 and \S 4. The third case is a bit more difficult.  

Let $W$ be an $H$-module where $H$ is reductive. Set $Y:=\quot WH$ and let $\pi\colon W\to Y$ be the quotient morphism. Let $\Diff^k(W)$ (resp.\ $\Diff^k(Y)$) denote the algebraic differential operators on $W$ (resp.\  $Y$) of order at most $k$ (see \cite[\S 3]{SchLiftingDOs}). Then we have a morphism $\pi_*\colon\Diff^k(W)^H\to\Diff^k(Y)$ where $\pi_*(P)(f)=P(\pi^*(f))$ for $P\in \Diff^k(W)^H$ and $f\in\O(Y)\simeq\O(W)^H$. 

\begin{definition}\label{defn:admissible}
We say that the $H$-module $W$ is \emph{admissible\/} if
\begin{enumerate}
\item $W$ is $2$-principal.
\item  $\pi_*\colon\Diff^k(W)^H\to\Diff^k(Y)$ is surjective for all $k$.
\end{enumerate}
\end{definition}

From \cite[Theorem 2.2]{SchAutoms} we have the following result:

\begin{theorem}
If $W$ is admissible, then $W$ is good.
\end{theorem}
\noindent (\cite[Theorem 2.2]{SchAutoms} is actually stated for $\HAut(Y)$, but the proof for $\HAut(Y,0)$ is the same.)
In sections \S 5 and \S 6 we show that $(V_3,\widetilde G_3)$ is admissible, hence good. From \cite[1.9, 1.16]{SchAutoms} we see that $V_3$ has the lifting property. Thus each $(V_i,\widetilde G_i)$ satisfies the conclusions of Theorem \ref{thm:mainhelp}. In \S 7 we combine these facts to prove Theorems \ref{thm:main} and \ref{thm:mainhelp}. In \S 8 we consider the analogues of our results for compact Lie groups.

\smallskip\noindent
\textit{Acknowledgement.}  We thank the referee for a careful reading of the manuscript and for comments improving the exposition.

\section{Preserving the stratification}
Let  $H$ be reductive and $W$ an  $H$-module. Let $Y=\quot WH$. For $(L)$ and $(M)$ conjugacy classes of subgroups of $H$ we write $(M)\geq (L)$ if $L$ is conjugate to a subgroup of $M$. From  \cite[Lemma 5.5]{SchLifting} one has

\begin{lemma}\label{lem:closure}
Let $S=Y_{(L)}$ be a stratum of $Y$. Then  $S$ is irreducible and
$$
\overline{S}=\bigcup_{(M)\geq (L)} Y_{(M)} 
$$
\end{lemma}

Let $\X(Y)$ denote the strata preserving vector fields on $Y$. These are the derivations of $\O(Y)$ which are tangent to $S$ along $S$ for every stratum $S$ of $Y$. Equivalently, they are the derivations preserving the ideals vanishing on the closures of the strata of $Y$. Let $\X_h(Y)$ denote the holomorphic strata preserving vector fields on $Y$ and let $\X_h(Y,0)$ denote the corresponding germs at $0$. Let $\X(W)$ denote the vector fields on $W$. Then it is easy to see that any element of $\X(W)^H$ preserves the pull-backs of the ideals vanishing on the closures of the strata of $Z$ \cite[Corollary 1.3]{SchLifting}. Thus  we have a natural morphism $\pi_*\colon \X(W)^H\to\X(Y)$.
We also have $\pi_*\colon\X_h(W)^H\to\X_h(Y)$ and similarly for $\X_h(W,0)^H$ and $\X_h(Y,0)$.  
Now we assume that $W$ is orthogonal.

\begin{proposition}\label{prop:codimone}
Let  $H$ be reductive and $W$ an orthogonal $H$-module.  
\begin{enumerate}
\item  Let $A$ be a derivation of $\O(Y)$ which preserves the codimension one strata of $Y$. Then $A$ preserves all the strata of $Y$. The same result holds for  holomorphic derivations of $\HH(Y)$ and germs of holomorphic derivations  at $0\in Y$ \cite[3.5, 5.8, 6.1]{SchLifting}.
\item Let $A\in\X(Y)$ $($resp. $\X_h(Y)$, resp.\ $\X_h(Y,0))$. Then there is a $B\in \X(W)^H$ $($resp.\ $\X_h(W)^H$, resp.\ $\X_h(W,0)^H)$ such that $\pi_*B=A$ \cite[3.7, 6.7, 6.9]{SchLifting}.
\end{enumerate}
\end{proposition} 

\begin{corollary}\label{cor:preservestrata}
Let $\psi\in\Aut(Y)$ such that $\psi$ preserves the union of the codimension one strata. Then $\psi$ preserves the stratification of $Y$.
\end{corollary}

\begin{proof}
 We follow the argument of \cite{SchVectorFields}.
Let $A\in\X(Y)$ and define $\psi^*(A)=\psi^*\circ A\circ(\psi^*)\inv$. This is a derivation of $\O(Y)$ which, by Proposition \ref{prop:codimone}, again lies in $\X(Y)$. Thus $\psi^*$ preserves $\X(Y)$. One computes that, in terms of tangent vectors, $(\psi^*A)(y)=d(\psi\inv)_{\psi(y)}A(\psi(y))$, $y\in Y$.  Let $s\in S$ where $S$ is a  stratum of $Y$.  Then the evaluation at $s$ of all elements of $\pi_*\X(W)^H=\X(Y)$ is precisely $T_sS$. This is established in  \cite[Proposition 2.2]{SchVectorFields}. Now the dimension of $\X(Y)=\psi^*\X(Y)$ evaluated at $s$ is  the same as the dimension of $\X(Y)$ evaluated at $\psi(s)$. Hence the stratum $S'$ containing $\psi(s)$ has the same dimension as $S$. Since there are finitely many strata and the strata are irreducible, we see that $\psi(S)=S'$. Hence $\psi$ permutes the strata of $Y$.
\end{proof}

\begin{remark}\label{rem:morecases}
The same argument works to show that if $\psi\in\HAut(Y)$ or $\HAut(Y,0)$ and $\psi$ preserves the union of the codimension one strata, then it preserves the stratification.
\end{remark}

\section{Copies of $2\liesl_2$} 
 
Let $(W_i,H_i)=(2\liesl_2,\PSL_2)\simeq(2\C^3,\SO_3)$, $i=1,\dots,n$. Let  $W=\oplus W_i$ have the product action of $H=\prod_i H_i$ with quotient $Y$.
Let $Y_i$ denote $\quot {W_i}{H_i}$. Then $Y=Y_1\times\dots\times Y_n$ where each $Y_i$ is isomorphic to $\C^3$, and the non principal strata of $Y_i$ are the nonzero points in a cone $\Var(xy-z^2)$ (denoted $S_i$) and the origin of $\C^3$. Let $S_i'$ denote the product of $S_i$ and the principal strata of the other $Y_j$. The $S_i'$ are the codimension one strata of $Y$. 

\begin{proposition}\label{prop:sl2good}
The $H$-module $W$ is good and has the lifting property.
\end{proposition}

\begin{proof}
Let $\psi\in\HAut(Y)$ preserve the   codimension one strata.   Then $\psi$ preserves the stratification. The point $0$ is the lowest dimensional stratum, hence it is preserved. Since each $\O(W_i)^{H_i}$ is generated by quadratic invariants, $\psi_0$ exists and is just $d\psi(0)$, considered as a mapping of $Y\simeq T_0(Y)$ to itself.  The same argument works in case that $\psi\in\HAut(Y,0)$. Hence $W$ is good.

Now let $\psi$ be quasilinear and strata preserving. Then $\psi$ permutes the  $S_i'$ and modulo a permutation of the $W_i$ (which induces an automorphism $\sigma$ of $H$), we may reduce to the case that $\psi$ preserves each $ S_i'$. Now $\overline{S_i'}$  is the product of $\overline{S_i}$ with the other $Y_j$. The singular points of this space are the product of $0\in Y_i$ with the product of the other $Y_j$. By taking intersections we see that $\psi$ preserves the product of $Y_i$ with the origin of the other $Y_j$. Thus  $\psi=\diag(A_1,\dots,A_n)$ where $A_i\in\GL(3)$ preserves the cone $\overline{S_i}$, $i=1,\dots,n$. But then $A_i$ is in $\C^*$ times the orthogonal group of $xy-z^2$. This is the image of the group $GL_2$ which acts on $2\liesl_2\simeq\liesl_2\otimes\C^2$ via its action on $\C^2$, commuting with the action of   $\PSL_2$. Hence $\psi$ lifts and $W$ has the lifting property.
\end{proof}
 
 \section{The adjoint representation}
 
We now consider the case where $(W,H)=(\lieh,H)$ and $H$ is semisimple.   We have $\lieh=\oplus_{i=1}^n\lieh_i$ where $\lieh_i$ is simple. Let $H_i$ denote the corresponding adjoint group so that  $H=\prod_i H_i$. Let $\liet_i$ be a maximal toral subalgebra of $\lieh_i$ and let $\W_i$ be the corresponding Weyl group. Set $\liet=\oplus_i\liet_i$ and $\W=\prod_i \W_i$. Let $Y_i$ denote $\quot {\lieh_i}{H_i}\simeq\liet_i/\W_i$. Define $D_s$ and $D_\ell$ as in the introduction.  
  
Let $\psi\in\HAut(Y,0)$ preserve  $D_s$ and $D_\ell$.   First we only use the information that $\psi$ preserves $D_s\cup D_\ell$.  
Let $\tilde\pi$ denote the quotient mapping $\liet\to Y=Y_1\times\dots\times Y_n$.

\begin{proposition}
Let $\psi$ be as above. Then 
\begin{enumerate}
\item $\psi$ is stratification preserving. In particular, $\psi(Y_\pr)\subset Y_\pr$. 
\item $\psi$ has a lift $\tau\in\HAut(\liet,0)$ which is $\sigma$-equivariant for some $\sigma\in\Aut(\W)$.
\item $\psi$ is deformable and $\psi_0$ lifts to $\tau'(0)\in N_{\GL(\liet)}(\W)$.
\item $W$ is good.
\end{enumerate}
\end{proposition}

\begin{proof}
Part (1) follows from Corollary \ref{cor:preservestrata}. Then   \cite{Lyashko}, \cite[Theorem 5.4]{KrieglTensor} or \cite[Theorem 3.1]{SchIsoPreserveInvars} show that $\psi$ has a $\sigma$-equivariant lift $\tau$ to $\liet$, giving (2). Clearly $\tau(0)=0$. Let $x\in\liet$ and set  $y=\tilde\pi(x)$. Then  $\lim_{t\to 0}\psi_t(y)=\tilde\pi(\lim_{t\to 0}t\inv\tau(tx))=\tilde\pi(\tau'(0)x)$. Thus $\psi$ is deformable and we have (3) and (4).
\end{proof}

We are not done yet. If $\psi$ is quasilinear, then we have a linear lift $\tau$ which is $\sigma$-equivariant. We have to lift $\tau$ to $\lieh$. We  have not used yet that $\psi$ preserves $D_s$. Let $\Phi\subset \liet^*$ be the roots of $\lieh$. For $\alpha\in\Phi$, let $\tau_\alpha$ denote the corresponding reflection in $\liet$. We have the Killing form $B_i$ on $\liet_i^*$. Since $\W_i$ acts irreducibly on $\liet_i^*$ for each $i$, any $\W$-invariant non degenerate bilinear form on $\liet^*$ has to be of the form $\oplus c_i B_i$ where $c_i\in \C^*$. 
Let $\tau\in N_{\GL(\liet)}(\W)$. Then $\tau$ permutes the factors $\liet_i$ so we have an action of $\tau$ on $\{1,\dots,n\}$ where  $\tau(\liet_i)=\liet_{\tau(i)}$ Let $\tau^*\colon\liet^*\to\liet^*$ be composition with $\tau$.  We have the push-forward $\tau_*(B_i)$ where $\tau_*(B_i)(\xi,\eta)=B_i(\tau^* \xi,\tau^* \eta)$ for $\xi$, $\eta\in\liet_{\tau(i)}^*$. Set  $\tau_*B=\oplus_i\tau_*B_i$ where $B=\oplus_i B_i$ is the Killing form.  We will also denote $B$ by $(\, ,\,)$.

 \begin{lemma}\label{lem:modscalar}
We can modify $\tau|_{\liet_i}$ by a scalar $d_i$, $1\leq i\leq n$, such that $\tau_*B=B$.
 \end{lemma}
 
 \begin{proof}
Since $\tau$ normalizes $\W$, $\tau_*B$ is invariant under $\W$, hence is of the form $\oplus c_iB_i$.
 Since $\tau$ acts on $\{1,\dots,n\}$ by a product of cycles, one easily sees that one can modify the restriction of $\tau$ to $\liet_i$ by a scalar $d_i$ such that $\tau_*B_i=B_{\tau(i)}$, $1\leq i\leq n$. Note that this  modification  lifts to $\GL(\lieh)^H\simeq(\C^*)^n$.
 \end{proof}
 
  \begin{lemma}
  Assume that $\tau$ preserves the Killing form.
 Then for all $\alpha\in\Phi$, $\tau^*(\alpha)=c(\alpha)\phi(\alpha)$ where $\phi\colon\Phi\to\Phi$ and $c(\alpha)>0$.
 \end{lemma}

 \begin{proof}
 The composition  $\tau\inv\circ\tau_\alpha\circ\tau$ is a reflection in $\W$, hence it is  $\tau_\beta$ for some $\beta\in\Phi$. Thus $\tau^*\alpha=c\beta$ for some $c\in\C^*$. Since $\tau$ preserves the Killing form,  $c^2(\beta,\beta)$ is positive, hence $c\in\R^*$. Perhaps changing $\beta$ to $-\beta$, we can arrange that $c>0$.  Set $\phi(\alpha)=\beta$ and $c(\alpha)=c$.
 \end{proof}
 
 \begin{theorem}\label{thm:centralizer}
Let $\tau\in N_{\GL(\liet)}(\W)$ be a  lift of $\psi\in\Aut_\ql(Y)$ where $\psi$ is strongly stratification preserving. Then $\tau$ lifts to $N_{\GL(\lieh)}(H)$. Hence $W$ has the lifting property.
 \end{theorem}
 
 \begin{proof} We may assume that $\tau$ preserves the Killing form and induces a bijection $\phi$ of $\Phi$ where $\tau^*(\alpha)=c(\alpha)\phi(\alpha)$ and $c(\alpha)>0$, $\alpha\in\Phi$. Thus $\tau^*$ permutes the reflection hyperplanes in $\liet^*_\R$, the real span of the roots. Hence $\tau^*$ permutes the chambers  and there is an element $w\in \W$ such that $w\tau^*$ preserves the Weyl chamber $C(\Delta)$ where $\Delta$ is a base of $\Phi$. Hence we may assume that $\tau^*$   preserves $C(\Delta)$ and it follows that $\phi$ preserves $\Delta$.
 
Let $i\in\{1,\dots,n\}$ and set $j=\tau(i)$. Let $\alpha$, $\beta\in\Delta_j$ where $\Delta_j$ is the base of $\Phi_j\subset\liet_j^*$.  
Let $\langle \beta,\alpha\rangle$ denote $2(\beta,\alpha)/(\alpha,\alpha)$. Since $\tau^*$ is an isometry we have 
$$
\langle\beta,\alpha\rangle\langle\alpha,\beta\rangle=\langle\tau^*(\beta),\tau^*(\alpha)\rangle\langle \tau^*(\alpha),\tau^*(\beta)\rangle= \langle\phi(\beta),\phi(\alpha)\rangle\langle\phi(\alpha),\phi(\beta)\rangle.
$$
Thus $\phi\colon \Delta_{j}\to\Delta_i$ preserves the product $\langle\beta,\alpha\rangle\langle\alpha,\beta\rangle$. Since $\tau$ preserves the collection of reflection hyperplanes for the short roots, we have that $|\alpha|<|\beta|$ if and only if $|\phi(\alpha)|<|\phi(\beta)|$ where $|\alpha|$ is the norm of $\alpha$, etc. Now the $\langle\beta,\alpha\rangle$   are negative (or zero) and one can   see that the numbers $\langle\beta,\alpha\rangle$ are preserved by $\phi$ (see \cite[Ch.\ III, Table 1]{HumphLieAlg}). Hence the restriction of $\phi$ to $\Delta$ is a diagram automorphism. Now changing $\tau$ by a diagram automorphism, which lifts to $N_{\GL(\lieh)}(H)$, we can reduce to the case that $\phi$ is the identity on $\Delta$. Thus $\tau^*$ preserves the $\liet_i^*$ and sends $\alpha$ to $c(\alpha)\alpha$, $\alpha\in\Delta_i$, where $c(\alpha)>0$. Hence $\tau^*$ commutes with the simple reflections on $\liet^*$, so it centralizes $\W$.  Since each $\W_i$ acts irreducibly on $\liet_i^*$, $\tau^*$ must be multiplication by a (positive) scalar on each $\liet_i^*$. Since  $\tau^*$ is an isometry, we see that we have reduced to the case that $\tau^*$ is the identity.
\end{proof}
  
\section{Modularity}

We want to  prove that $(V,G)$ is admissible if it is large.
 We first need to recall some notions and results from \cite{SchDiffSimple, SchLiftingDOs}. Let $H$ be a linear algebraic group and $X$ an $H$-variety. Let $X_{(n)}$ or $(X,H)_{(n)}$ denote the  set of points of $X$ with isotropy group of dimension $n$. Then $X_{(n)}$ is constructible and we define the \emph{modularity of $X$\/}, $\mod(X,H)$, by
 $$
 \mod(X,H)=\sup_{n\geq 0}(\dim X_{(n)}-\dim H+n)
 $$
 where, as usual, the dimension of the empty set is $-\infty$. If $X_{(n)}$ is irreducible, then $\dim X_{(n)}-\dim H+n$ is the transcendence degree of $\C(X_{(n)})^H$. We have that
$$
\mod(X,H)=\dim X-\dim H+\sup_{n\geq 0}(n-\codim X_{(n)}).
$$

 Now suppose that  $X_{(0)}\neq\emptyset$. We say that $X$ is \emph{$k$-modular\/} if $\mod(X\setminus X_{(0)},H)\leq \dim X-\dim H-k$. Equivalently, $\codim X_{(n)}\geq n+k$ for all $n\geq 1$. Let us assume from now on that $H$ is reductive and that  $X$ is an affine $H$-variety with FPIG. Then $X_{(0)}\neq\emptyset$ and $\dim X-\dim H=\dim\quot XH$ so that $X$ is $k$-modular if and only if $\mod(X\setminus X_{(0)},H)\leq\dim \quot XH-k$.
 
 There are geometric interpretations of $X$ being $0$-modular. From now on assume that $X$ is smooth. Let $A_1,\dots,A_r$ be a basis of the Lie algebra of $H$. We may consider the $A_i$ as vector fields $\tilde A_i$ on $X$, hence as functions $f_{A_i}$ on $T^*X$. Then $X$ is $0$-modular if and only if the $f_{A_i}$ are a regular sequence in $\O(T^*X)$, i.e., their set of common zeroes has codimension $r$. Another interpretation is the following. Let $\mu\colon T^*X\to\lieh^*$ be the moment mapping which sends $\xi\in T^*_xX$ to the element of $\lieh^*$ whose value on $A\in\lieh$ is $\tilde A(x)(\xi)$. Then $X$ is $0$-modular if and only if all the fibers of $\mu$ have codimension $r$ \cite[Remark 8.6]{SchLiftingDOs}.

 \begin{remark}\label{rem:mod=dim}
 Note that if $X$ is $k$-modular for any $k\geq 0$, then $\mod(X,H)=\dim  X-\dim H$.
 \end{remark}
 
Suppose that $X$ is   $2$-principal. Then we can find invariant functions $h_1$ and $h_2$ which vanish on $X\setminus X_\pr$ and are a regular sequence in $\O(X)$. If $X$ is also $2$-modular, then $h_1$, $h_2$, $f_{A_1},\dots,f_{A_r}$ is a regular sequence in $\O(T^*X)$  \cite[Lemma 9.7]{SchLiftingDOs}. Then we can apply  \cite[Proposition 8.15]{SchLiftingDOs}:

\begin{proposition}
Suppose that there are  $h_1$, $h_2$ as above such that $h_1$, $h_2$ and the $f_{A_i}$  are  a regular sequence in $\O(T^*X)$. Then $\pi_*\colon\Diff^k(X)^H\to\Diff^k(\quot XH)$ is surjective for every $k\geq 0$.
\end{proposition} 

\begin{corollary}\label{cor:main}
Suppose that $X$ is $2$-principal and $2$-modular. Then $\pi_*\colon\Diff^k(X)^H\to\Diff^k(\quot XH)$ is surjective for every $k\geq 0$.  
\end{corollary}

In the next section we will show that our $V$ of interest is $2$-modular. We will need the following result. Recall that a group action is \emph{almost effective\/} if the kernel of the action is finite.

  \begin{lemma}\label{lem:vinberg}
  Let $H$ be connected  reductive. Let $T$ denote the connected center of $H$ and let $H_s$ denote $[H,H]$. Let $W:=W'\oplus W''$ be a direct sum of  $H$-modules where $T$ acts trivially on $W'$. Assume that the action of $T$ on $W''$ is almost effective. If $(W',H_s)$ is $0$-modular, then so is $(W,H)$.
    \end{lemma}
 
\begin{proof} 
By Vinberg \cite{VinbergComplexity}, for every $q\geq 0$, $\codim (W'',T)_{(q)}\geq q$.  Let $w=(w',w'')$ where $w'\in W'$ and $w''\in(W'',T)_{(q)}$. If $\dim H_w=p$, then $p\geq q$ and  $\dim (H_s)_{w'}=p'-q$ for $p'\geq p$. Thus
\begin{equation*}
\begin{aligned}&\sup_{p \geq q}(p-\codim ((W'\times (W'',T)_{(q)},H)_{(p)})\\
\leq &\sup_{p'\geq q}(p'-q -\codim ((W',H_s)_{(p'-q)})+q-\codim(W'',T)_{(q)}\\
\leq&\sup_{p'\geq 0}(p'-\codim(W',H_s)_{(p')})\leq 0.
\end{aligned}
\end{equation*}
Since $q\leq p$ is arbitrary, we see  that $(W,H)$ is $0$-modular.
\end{proof}

\section{Large representations.}

 We assume that  $W=\oplus_{i=1}^n r_i\lieh_i$ is large. Recall that $W$ is $2$-principal. By Corollary \ref{cor:largelifting} a quasilinear $\psi\in\Aut_\ql(\quot WH)$ has a biholomorphic $\sigma$-equivariant lift $\Psi$. Then $\Psi'(0)$ is also a lift, so that $W$ has the lifting property. Thus, by Corollary \ref{cor:main}, we only need to show that $W$ is $2$-modular. It is easy to see that  $W$ is $2$-modular if and only if each $r_i\lieh_i$ is $2$-modular, so we may reduce to the case that $H$ is simple. The case $H=\PSL_2$ follows from \cite[Theorem 11.9]{SchLiftingDOs}, so we only have to consider the case where the rank of $H$ is at least two. Moreover, we need only consider the case of two copies of $\lieh$ since increasing the number of copies can only increase the modularity \cite[Proposition 11.5]{SchLiftingDOs}.

Let $v \in W \setminus W_{(0)}$. Then $\dim H_v > 0$ and
there is an injective homomorphism $\lambda : L \to H_v$ where $L$ is a copy of
the multiplicative group $\C^*$ or the additive group $\C^+$.  First consider
the case where $L = \C^*$. Acting by an element of $H$, we may assume that
$\Im \lambda \subseteq T$, where $T$ denotes a maximal torus of $H$.  Identify $L$ with its image in $T$.  Since the action of $T$ on $W$ is diagonalizable and $L\subset T$, there are
only finitely many possible $W^L$. Let $C_L$ denote $C_H(L)$.
Then $H\cdot W^L$ is constructible and the $H$-orbits there all intersect $W^L$ in a union of $C_L$ orbits. Hence $\mod(H\cdot W^L,H)\leq\mod(W^L,C_L)$.

 Now suppose that $L = \C^+$.  By the Jacobson-Morozov theorem,
$\lambda$ extends to an injective homomorphism, also called $\lambda$, from $\SL_2$ or $\PSL_2$ to 
$H$. We identify $L$ with its image in $H$ and we let $S$ denote $\lambda(\SL_2)$ or $\lambda(\PSL_2)$ as the case may be.
Up to conjugation in $H$,  there are only finitely many possible $\lambda$. Let $C_L$ denote $C_H(S)\times T''$ where $T''$ is the maximal torus of $S$. Then again we have the estimate
that $\mod(H\cdot W^L,H)\leq\mod(W^L,C_L)$.

Since the orbit of every  point in $W\setminus W_{(0)}$ intersects one of our finite collection of sets $W^L$, and since $\dim \quot WH=\dim H$, our discussion above gives the following.

\begin{lemma}\label{lem:reducelambda}
Suppose that $\mod(W^L,C_L)\leq\dim H-2$ for all   $L$ of dimension one. Then $W$ is $2$-modular.
\end{lemma}

 The next two propositions show that we always have $\mod(W^L,C_L)\leq\dim H-2$, hence they complete our proof that $W$ is admissible, hence good.

 \begin{proposition}
 Suppose that $L=\C^*$. Then $\mod(W^L,C_L)\leq\dim H-2$.
 \end{proposition}
     
 \begin{proof}
As above,  we may assume that $L$ is contained in a maximal torus $T$ of $H$.   The action of $C_L$ on $W^L$ is twice the adjoint representation of  $C_L$. A maximal torus of $C_L$ is $T$, and  the connected center of $C_L$ is a subtorus $T'$ of $T$ containing $L$. Then $C_L/T'$ is semisimple, so its action on $W^L$ is a trivial representation plus the sum of twice the adjoint representation of each simple component $M_i$ of $C_L/T'$. By induction on rank, $2\lie m_i$ is $2$-modular if $M_i$ has rank at least two. If $\lie m_i\simeq\liesl_2$, then   $2\lie m_i$ is $1$-modular. It then follows from Remark \ref{rem:mod=dim} that 
 $$
 \mod(W^L,C_L) =\dim W^L-\dim C_L+\dim T' =\dim \lieh^L+\dim T'.
 $$
Let $\alpha_1,\dots,\alpha_k$ be the positive roots of $\lieh$ which are nontrivial when restricted to $T'$. As $T'$-module, 
$$
\lieh=\theta_m+\bigoplus^k_{i=1} ( \lieh_{\alpha_i}+\lieh_{-\alpha_i})
$$ 
where $m=\dim \lieh^L$ and $\theta_m$ denotes the trivial module of dimension $m$.   We need to show that $\dim\lieh-2-(m+\dim T')$ is nonnegative, i.e., that  $-2-\dim T'+2k\geq 0$. Since the action of $T'$ on $\lieh$ is effective, $k\geq \dim T'$. Thus, if $\dim T'\geq 2$, we obtain the desired estimate. The only problem is the case that $T'=L$ has dimension one.  Suppose that only one positive root $\alpha$ of $H$ acts nontrivially on $L$. We may assume that $\alpha$ is simple. Then, since $H$ has rank at least two, there is a simple root $\beta$ (acting trivially on $L$) such that $\alpha+\beta$ is a root. Then $\alpha+\beta$ acts nontrivially on $L$ and we have a contradiction. Hence $k\geq 2$ and  $\mod(W^L,C_L)\leq\dim H-2$. 
  \end{proof}

 \begin{proposition}
 Suppose that $L=\C^+$. Then $\mod(W^L,C_L)\leq\dim\quot WH-2$.
 \end{proposition}

 \begin{proof}
Let $L\subset S\subset H$ be as discussed above.  As $S$-module,   $\lieh=\sum_i m_i R_i$ where $R_i$ denotes the space of binary forms of degree $i$ and $m_i\geq 0$ is the multiplicity of $R_i$. The fixed points of $L$ in each $R_i$ have dimension one, so that $W^L$ has dimension $2\sum_i m_i$. The  Lie algebra of $C_L$ has dimension $m_0+1$ and the action of $C_H(S)$ on   $2\lieh^S$ is twice its adjoint representation. 
Write $C_H(S)^0=C_s T'$ where $C_s$ is semisimple and $T'$ is the connected center. As we saw above,  if $C_s$ is nontrivial, its action on $2\lieh^S$ is $1$-modular. Recall that $T''$ is the maximal torus of $S$.   Now $T'$ acts on $m_iR_i$ via a homomorphism to $\GL_{m_i}$, hence $T'$ acts almost effectively on $\oplus_{i\geq 1} m_iR_i^L$. Since $T'$ commutes with $S$, it acts trivially on    $\lie s\subset m_2 R_2$. However, $T''$ acts nontrivially on $\lie s^L$. Thus $T'\times T''$ acts almost effectively on  $W^L$.   Using Lemma \ref{lem:vinberg}, we get that $\mod(W^L,C_L)$ is $\dim W^L-\dim C_H(S)-1$. Now $\dim W^L$ is $2(m_0+\sum_{i\geq 1}m_i)$ and $\dim C_H(S) =m_0$. We need that $\dim W^L-m_0-1\leq\dim\lieh-2$. Hence we need 
 $$
 2(m_0+\sum_{i\geq 1} m_i)-m_0-1\leq\sum_{i\geq 0}(i+1)m_i-2\text{, i.e., }0\leq  -1+\sum_{i\geq 2} (i-1)m_i.
 $$
Since $\lie s\simeq R_2\subset\lieh$, we have that $m_2\geq 1$, giving us the desired inequality.
  \end{proof}
 
\section{Proofs of the main theorems}

\begin{lemma}\label{lem:rank2}
Let $H$ be a simple adjoint group of rank at least two. Let $S$ be a codimension one stratum of $Y:=\quot\lieh H$. Then $\overline{S}$ contains a codimension two stratum of $Y$.
\end{lemma}

\begin{proof}
Let $\liet$ be a maximal toral subalgebra of $\lieh$. Then there is a simple root $\alpha$ such that $\overline{S}$ is the image in $Y$ of the corresponding root hyperplane $\liet_\alpha$. Let $\beta$ be another  simple root  and let $S'$ denote the corresponding stratum of $Y$. Then $\overline{S}\cap\overline{S'}$ is the image of $\liet_\alpha\cap\liet_\beta$ and contains a stratum of codimension two.
\end{proof}

 Let $V=\oplus_ir_i\lieg_i$ and $G$ be as in Theorem \ref{thm:mainhelp}. We have decomposed $V$ as $V_1\oplus V_2\oplus V_3$ and $G$ as $\widetilde G_1\times \widetilde G_2\times \widetilde G_3$ where $(V_1,\widetilde G_1)$ is a direct sum of representations isomorphic to $(2\liesl_2,\PSL_2)$, $(V_2,\widetilde G_2)$ is the adjoint representation of a semisimple group and $(V_3,\widetilde G_3)$ is large. Then $Z=Z_1\times Z_2\times Z_3$ where $Z=\quot VG$ and $Z_i=\quot{V_i}{\widetilde G_i}$, $1\leq i \leq 3$. We have shown that each $(V_i,\widetilde G_i)$ is good and has the lifting property. Now we have to show the same for $(V,G)$.
 
 For $S_i$ a stratum of $Z_i$, let $S_i'$ denote the product of $S_i$ and the principal strata of the other two factors of $Z$, $1\leq i\leq 3$.
Let $\psi\in\HAut(Z,0)$ be stratification preserving and write $\psi=(\psi_1,\psi_2,\psi_3)$ where $\psi_i$ is a germ sending $Z$ to $Z_i$, $1\leq i\leq 3$.
 \begin{lemma}\label{lem:together}
\begin{enumerate}
\item $\psi$ preserves  the strata of the form $R_3'$ and the  sets of the form $Z_1\times Z_2\times R_3$ where $R_3$ is a stratum of $Z_3$.
\item The analogues of (1) hold for strata $R_1$ of $Z_1$ and $R_2$ of $Z_2$.
\item $d\psi(0)$ has the form $\diag(A_1,A_2,A_3)$ where $A_i\colon T_0(Z_i)\to T_0(Z_i)$ is an isomorphism, $1\leq i\leq 3$.
\item For $z$ near $0$, $\psi_3(z_1,z_2,z_3)$ is a strata preserving automorphism of $Z_3$, fixing $0$, parameterized by $(z_1,z_2)\in Z_1\times Z_2$.
\item $\psi_{3,0}(z)=\lim_{t\to 0}t\inv\cdot \psi_3(t\cdot z)=\lim_{t\to 0}t\inv\cdot\psi_3(0,0,t\cdot z_3)$ exists. Thus $\psi_{3,0}(z)=\psi_{3,0}(0,0,z_3)$ can be considered as an element of  $\Aut_\ql(Z_3)$.
\item The analogues of (4)--(5) hold for $\psi_1$ and $\psi_2$.
\item $\psi$ is deformable and $\psi_0=(\psi_{1,0},\psi_{2,0},\psi_{3,0})$ where $\psi_{i,0}\in\Aut_\ql(Z_i)$ is strata preserving, $1\leq i\leq 3$.
\end{enumerate}
  \end{lemma}

\begin{proof}
Let $R_3$ be a stratum of $Z_3$. Then $\psi(R_3')$ is  of the form $S_1\times S_2\times   S_3$ where $S_i$ is a stratum of $Z_i$, $1\leq i \leq 3$. If $S_1$ is not $Z_{1,\pr}$, then $\psi(R_3')$ is in the closure of a codimension one stratum, hence so is $R_3'$. Since $Z_3$ has no codimension one strata, we get a contradiction. Thus $S_1=Z_{1,\pr}$ and similarly $S_2=Z_{2,\pr}$. Hence $\psi$ permutes the strata of the form $R_3'$ and all sets of the form $Z_1\times Z_2\times\overline{R_3}$. By Lemma \ref{lem:closure} we see that $\psi$ preserves sets  of the form $Z_1\times Z_2\times R_3$, giving (1).

 Let $S_1$ be a codimension one stratum of $Z_1$. Then $S_1'$ is the codimension one stratum $R$ of a copy of $Y:=\quot {(2\liesl_2)}{\PSL_2}$ times the principal stratum  of the other factors of $Z$. Moreover,  $\overline{R}$ is $R$ union a point which has codimension three in $Y\simeq\C^3$.  Now consider the case where we have a codimension one stratum $S_2$ of $Z_2$. Then $S_2'$ is the codimension one stratum $R$ of a copy of $Y:=\quot{\lie l}L$, where $L$ is simple, times the principal stratum of the other factors of $Z$. If $\psi$ sends $S_1'$ to  $S_2'$, then the corresponding space $\quot{\lie l}L$ has dimension at least three and contains no codimension two strata. Thus $\rank L>1$ and we get a contradiction to  Lemma \ref{lem:rank2}.   It follows that $\psi$ permutes all strata of the form $S_1'$  so $\psi$ permutes the sets of the form $S_1\times Z_2\times Z_3$. The analogous results hold for   strata $S_2$ of $Z_2$ and we have (2).

Clearly $d\psi_0$ has the form given in (3). Thus for $(z_1,z_2)$ near $0$, $\psi_3(z_1,z_2,z_3)$ is a germ of an automorphism of $Z_3$ which fixes $0$. For this one uses the inverse function theorem for varieties, see, for example, \cite[Lemma 14.15]{HeiSchCartan}. Here we don't need to worry about preserving strata since $V_3$ is large. Hence we have (4) and since $V_3$ is good, we have (5).  In the case of $\psi_1$, we note that $\psi$ respects the stratification by strata $S_1\times Z_2\times Z_3$ where $S_1$ is a stratum of $Z_1$. Thus $\psi_1(z_1,z_2,z_3)$, near $0$, is a family of strata preserving automorphisms of $Z_1$, fixing $0\in Z_1$, with parameters $(z_2,z_3)\in Z_2\times Z_3$. The analogous result holds for $\psi_2$ and since $V_1$ and $V_2$ are good, we have (6).  Part (7)  is immediate.
\end{proof}

\begin{proof}[Proof of Theorem \ref{thm:mainhelp}]
It follows from Lemma \ref{lem:together} that $V$ is good, and we also have that any $\psi\in\Aut_\ql(Z)$ is a product of elements $\psi_i\in\Aut_\ql(Z_i)$, $1\leq i \leq 3$,  where $\psi_2$ preserves the closures of the strata corresponding to short roots. Then it follows from Proposition \ref{prop:sl2good}, Theorem  \ref{thm:centralizer} and Corollary \ref{cor:largelifting} that $V$ has the lifting property.
\end{proof}

\begin{proof}[Proof of Theorem \ref{thm:main}]
By Theorem \ref{thm:mainhelp} we are in the following situation. We have $\psi\in \HAut(Z)$  and a deformation $\psi_t$ with $\psi_1=\psi$ and $\psi_0$  lifts to a $\sigma$-equivariant $\Psi_0\in\GL(V)$. Let $\widetilde\psi_t=\psi_t\circ(\psi_0\inv)$. Then $\widetilde\psi_t$ is an isotopy starting at the identity consisting of strata preserving automorphisms. Thus it is obtained by integrating a holomorphic (complex)  time dependent strata preserving vector field $A(t,z)$. Since $\pi_*\colon\X_h(V)^G\to\X_h(Z)$ is surjective, we can lift $A$ to a time dependent invariant  holomorphic vector field $B(t,v)$. By   \cite[Theorem 3.4]{SchAutoms} we can integrate $B$ for $0\leq t\leq 1$ to get  $G$-equivariant biholomorphic  lifts $\widetilde\Psi_t$ of $\widetilde \psi_t$. Then $\widetilde\Psi_1\circ\Psi_0$ is a $\sigma$-equivariant  biholomorphic lift of $\psi$.  
\end{proof}
 
 \section{Compact Lie groups}
 Let $K_i$, $i=1,\dots,d$, be simple compact adjoint Lie groups and consider the natural action of $K:=\prod_i K_i$ on $W:=\oplus r_i\liek_i$ where $r_i\in\N$, $1\leq i\leq d$. 
  Let $\pi\colon W\to W/K$ denote the quotient mapping. We can put a smooth structure on   $W/K$ (see \cite{SchSmooth}), as follows. A function on $W/K$ is smooth if it pulls back to a  (necessarily $K$-invariant) smooth function on $W$. We can make this more concrete, as follows. Let $p_1,\dots,p_r$ be homogeneous generators of $\R[W]^K$ and let $p=(p_1,\dots,p_r)\colon W\to\R^r$. Now $p$ is proper and separates the $K$-orbits in $W$. Let $X$ denote $p(W)\subset\R^r$. Then $X$ is a closed semialgebraic set and $p$ induces a homeomorphism $\bar p\colon W/K\to X$. The main theorem of \cite{SchSmooth} says that $p^*C^\infty(X)=C^\infty(W)^K$ where a function $f$ on $X$ is smooth if it extends to a smooth function in a neighborhood of $X$. Now we can define the notion of a smooth mapping $X\to X$ (or $W/K\to W/K$) in the obvious way. We also see an $\R^*$-action on $X$ where $t\cdot(y_1,\dots,y_r)=(t^{e_1}y_1,\dots,t^{e_r}y_r)$ for $t\in\R^*$, $(y_1,\dots,y_r)\in X$ and $e_j$   the degree of $p_j$, $j=1,\dots,r$.   We denote by $\Aut_\ql(W/K)$ the quasilinear automorphisms of $W/K$, i.e., the automorphisms which commute with the $\R^*$-action. 
 
  We have the usual stratification of $W/K$ determined by the conjugacy classes of isotropy groups of orbits.  There is a unique open  stratum, the principal stratum,  which consists of the orbits with trivial slice representation. It is connected and dense in $W/K$. Let  $E_s$ be the union of the closures of the codimension one strata of $W/K$ corresponding to the short roots.  We call a diffeomorphism $\Psi\colon W\to W$ \emph{$\sigma$-equivariant\/} if $\Psi\circ k=\sigma(k)\circ\Psi$, $k\in K$, where $\sigma\in\Aut(K)$.
  
  \begin{theorem}
  Let $\psi\colon W/K\to W/K$ be a diffeomorphism. Then $\psi$   lifts to a $\sigma$-equivariant diffeomorphism $\Psi\colon W\to W$ if and only if  $\psi$ preserves   $E_s$ 
  \end{theorem}

\begin{proof} 
Necessity of the condition on $\psi$ is established as before. Now assume that $\psi$ preserves $E_s$.
By \cite[Theorem 2.2]{LosikLift}, $\psi_0(x)=\lim_{t\to 0}t\inv\cdot\psi(t\cdot x)$ exists for every $x\in W/K$ where, of course, $\psi_0$ is quasilinear. By \cite[Theorem A]{BierstoneLifting}, $\psi_0$ (and $\psi$) necessarily preserve the principal stratum and the union of the codimension one strata. We want to show that $\psi_0$ has a $\sigma$-equivariant lift to $\GL(W)$.   Complexifying and applying Lemma \ref{lem:together} we see that it is enough to treat the following three cases:
\begin{enumerate}
\item $W$ is a direct sum of representations isomorphic to $(2\R^3,\SO(3,\R))$.
\item $W=\liek$ where $K$ is compact semisimple. Here we need that $\psi_0$ preserves $E_s$.
\item $W$ is large (defined as before).
\end{enumerate}
The proof of lifting in case (1) is as in Proposition \ref{prop:sl2good}. For case (3) one can use \cite[Proposition 8.5]{SchAutoms} which allows one to deduce the lifting property in (3) from Corollary \ref{cor:largelifting}. Note that (3) is not a trivial consequence of Corollary \ref{cor:largelifting}, since we have to show that the lift preserves $W\subset W\otimes_\R\C$. For case (2), a theorem of Strub \cite{StrubLocal} (see also \cite[Theorem 2.11]{SchIsoPreserveInvars}) shows that  there is a lift of $\psi_0$ to $\tau\in N_{\GL(\liet)}(\W)$ where $\liet$ is the Lie algebra of a maximal torus of $K$ and $\W$ is the Weyl group. Then the argument of Lemma \ref{lem:modscalar} through Theorem \ref{thm:centralizer}  shows that $\tau$ lifts to an element of $N_{\GL(\liek)}(K)$. Thus, if $\psi$ preserves $E_s$, there is a smooth deformation $\psi_t$  with $\psi_1=\psi$ such that $\psi_0$ lifts to $\Psi_0\in N_{\GL(W)}(K)$. By the isotopy lifting theorem of \cite{SchLifting} there is an   equivariant diffeomorphism  lifting   $\psi\circ\psi_0\inv$, hence there is a   $\sigma$-equivariant  diffeomorphism $\Psi$  lifting   $\psi$.  
\end{proof}
     %%%%%%%%%%%%%%%%%%
  
\def\cprime{$'$}
\providecommand{\bysame}{\leavevmode\hbox to3em{\hrulefill}\thinspace}
\providecommand{\MR}{\relax\ifhmode\unskip\space\fi MR }
\providecommand{\MRhref}[2]{%
  \href{http://www.ams.org/mathscinet-getitem?mr=#1}{#2}
}
\providecommand{\href}[2]{#2}


\begin{thebibliography}{KLM03}

\bibitem[Bie75]{BierstoneLifting}
Edward Bierstone, \emph{Lifting isotopies from orbit spaces}, Topology
  \textbf{14} (1975), no.~3, 245--252.

\bibitem[HS07]{HeiSchCartan}
Peter Heinzner and Gerald~W. Schwarz, \emph{Cartan decomposition of the moment
  map}, Math. Ann. \textbf{337} (2007), no.~1, 197--232.

\bibitem[Hum72]{HumphLieAlg}
James~E. Humphreys, \emph{Introduction to {L}ie algebras and representation
  theory}, Springer-Verlag, New York, 1972, Graduate Texts in Mathematics, Vol.
  9.

\bibitem[KLM03]{KrieglTensor}
Andreas Kriegl, Mark Losik, and Peter~W. Michor, \emph{Tensor fields and
  connections on holomorphic orbit spaces of finite groups}, J. Lie Theory
  \textbf{13} (2003), no.~2, 519--534.

\bibitem[Kra84]{KraftBook}
Hanspeter Kraft, \emph{Geometrische {M}ethoden in der {I}nvariantentheorie},
  Aspects of Mathematics, D1, Friedr. Vieweg \& Sohn, Braunschweig, 1984.

\bibitem[Kut11]{Kuttler}
J.~Kuttler, \emph{Lifting automorphisms of generalized adjoint quotients},
  Transformation Groups \textbf{16} (2011), 1115--1135.

\bibitem[Los01]{LosikLift}
M.~V. Losik, \emph{Lifts of diffeomorphisms of orbit spaces for representations
  of compact {L}ie groups}, Geom. Dedicata \textbf{88} (2001), no.~1-3, 21--36.

\bibitem[Lun73]{LunaSlice}
Domingo Luna, \emph{Slices \'etales}, Sur les groupes alg\'ebriques, Soc. Math.
  France, Paris, 1973, pp.~81--105. Bull. Soc. Math. France, Paris, M\'emoire
  33.

\bibitem[Lya83]{Lyashko}
O.~V. Lyashko, \emph{Geometry of bifurcation diagrams}, Current problems in
  mathematics, {V}ol. 22, Itogi Nauki i Tekhniki, Akad. Nauk SSSR Vsesoyuz.
  Inst. Nauchn. i Tekhn. Inform., Moscow, 1983, pp.~94--129.

\bibitem[Sch75]{SchSmooth}
Gerald~W. Schwarz, \emph{Smooth functions invariant under the action of a
  compact {L}ie group}, Topology \textbf{14} (1975), 63--68.

\bibitem[Sch80]{SchLifting}
\bysame, \emph{Lifting smooth homotopies of orbit spaces}, Inst. Hautes
  \'Etudes Sci. Publ. Math. (1980), no.~51, 37--135.

\bibitem[Sch94]{SchDiffSimple}
\bysame, \emph{Differential operators on quotients of simple groups}, J.
  Algebra \textbf{169} (1994), no.~1, 248--273.

\bibitem[Sch95]{SchLiftingDOs}
\bysame, \emph{Lifting differential operators from orbit spaces}, Ann. Sci.
  \'Ecole Norm. Sup. (4) \textbf{28} (1995), no.~3, 253--305.

\bibitem[Sch09]{SchIsoPreserveInvars}
\bysame, \emph{Isomorphisms preserving invariants}, Geom. Dedicata \textbf{143}
  (2009), 1--6.

\bibitem[Sch12]{SchAutoms}
\bysame, \emph{Quotients, automorphisms and differential operators},
  http://arxiv.org/abs/1201.6369 (2012).

\bibitem[Sch13]{SchVectorFields}
\bysame, \emph{Vector fields and {L}una strata}, J. Pure and Applied Algebra
  \textbf{217} (2013), 54--58.

\bibitem[Str82]{StrubLocal}
Rainer Strub, \emph{Local classification of quotients of smooth manifolds by
  discontinuous groups}, Math. Z. \textbf{179} (1982), no.~1, 43--57.

\bibitem[Vin86]{VinbergComplexity}
{\`E}.~B. Vinberg, \emph{Complexity of actions of reductive groups},
  Funktsional. Anal. i Prilozhen. \textbf{20} (1986), no.~1, 1--13, 96.

\bibitem[VP89]{PopovVinberg}
{\`E}.~B. Vinberg and V.~L. Popov, \emph{Invariant theory}, Algebraic geometry,
  4 ({R}ussian), Itogi Nauki i Tekhniki, Akad. Nauk SSSR Vsesoyuz. Inst.
  Nauchn. i Tekhn. Inform., Moscow, 1989, pp.~137--314, 315.

\end{thebibliography}
   \end{document}